\documentclass[11pt]{amsart}
\usepackage{amsfonts, amsmath, amssymb, color}
\usepackage{amsthm}
\usepackage{booktabs}
\usepackage{hhline}
\usepackage{lipsum}
\usepackage{lscape}
\usepackage{subfig}
\usepackage{textcomp}
\usepackage{tikz}
\usetikzlibrary{decorations.pathmorphing,shapes,arrows,positioning}
\usepackage{xcolor}
\usepackage{young}
\usepackage{youngtab}
\usepackage{ytableau}

\allowdisplaybreaks[4]
\theoremstyle{plain}
\newtheorem{thm}{Theorem}[section]
\newtheorem{theorem}{Theorem}[section]

\newtheorem{lemma}[thm]{Lemma}  
\newtheorem{prop}[thm]{Proposition}

\newtheorem{corollary}[thm]{Corollary} 

\theoremstyle{definition}
\newtheorem{defn}[thm]{Definition}

\newcommand{\la}{\lambda}

\numberwithin{equation}{section} \errorcontextlines=0

\def\al{\alpha}
\def\la{\lambda}
\def\ga{\gamma}
\def\ep{\epsilon}

\begin{document}
\title{Lattice structure of modular vertex algebras}

\author{Haihua Huang}
\address{School of Mathematics, South China University of Technology, Guangzhou, Guangdong 510640, China}
\email{hoiwaa@163.com}
\author{Naihuan Jing}
\address{Department of Mathematics, North Carolina State University, Raleigh, NC 27695, USA}
\email{jing@ncsu.edu}
\thanks{Supported in part by Simons Foundation grant No. 523868} 
\bigskip
\thanks{$*$Corresponding author: jing@ncsu.edu}

\maketitle

\begin{abstract}
In this paper we study the integral form of the lattice vertex algebra $V_L$.
We show that divided powers of general vertex operators preserve the integral lattice spanned by Schur functions indexed by partition-valued functions. We also
show that the Garland operators, counterparts of divided powers of Heisenberg elements in affine Lie algebras, also
preserve the integral form. These construe analogs of
the Kostant $\mathbb Z$-forms for the enveloping algebras of simple Lie algebras and the algebraic affine Lie groups
in the situation of the lattice vertex algebras.
\end{abstract}

\section{Introduction }  

The axioms of vertex algebras (VA), including the Jacobi identity \cite{B86, FLM}, make sense for any commutative ring, thus VA over $\mathbb Z$ can be defined naturally. An integral form of a vertex operator algebra (VOA) has been studied in \cite{BR, B98} for the Monster VOA and
in general in \cite{DG1}. As a precursor it was treated as the Kostant $\mathbb Z$-form of the enveloping algebra \cite{S} in the case of affine Lie algebras \cite{G} and their level one modules \cite{P}. For the most important examples of the lattice VOA and affine VA, a special integral form is spanned by monomials in Schur polynomials \cite{M} indexed either by lattice elements or by simple roots of the finite dimensional Lie algebra.
Modular Virasoro VAs and affine VAs have shown to carry finitely many irreducible modules over the field of finite characteristic \cite{JLM}, and other important properties for
modular VOAs have been studied in \cite{AW, GL, DR, LM, Mu}.

In the classical approach to finite algebraic groups, an important lattice property plays a crucial role: all divided powers of
integral Chevalley basis elements preserve the integral form of the underlying enveloping algebra, which ensures the introduction of
Lie groups of finite type over the characteristic $p$. This property was first studied by Garland \cite{G} for the affine Lie algebra,
and then was generalized to the basic module of the
affine Lie algebra of simply laced type \cite{P} (see also \cite{Mc}).
Lusztig has developed the $U_{\mathbf A}(\mathfrak g)$-lattice structure for studying quantum groups at roots of unity \cite{L}, where $\mathbf{A}=\mathbb{C}[q,q^{-1}]$.
First in \cite{J2} for affine $sl(2)$ then in general \cite{CJ}
for affine ADE types, it was proved that the level one modules of the quantum affine algebra $\mathbf{U}_q(\hat{\mathfrak g})$
 admits an integral form as a $\mathbf{U}_{\mathbf{A}}(\hat{\mathfrak g})$-lattice in the sense
 of Lusztig. In particular, the lattice is closed under quantum divided powers of basis elements of $\mathbf{U}_q(\hat{\mathfrak g})$, which then facilitates the study of the modules over the root of unity.
 The {\it goal} of the paper is to show that there is a similar lattice structure on a large class of
 vertex operator algebras and their irreducible modules.

Let $V=S(\hat{\mathfrak h}_{L})\otimes \mathbb C\{L\}$ be the lattice VOA associated with an integral lattice $L$
of rank $r$. Suppose $L$ has an integral basis $\{\alpha_i|i=1, \ldots, r\}$. This also includes the case of super-VOAs when
$L$ is not even. The symmetric algebra $S(\hat{\mathfrak h}_{L})$ can be viewed as a ring of symmetric functions
in the variables $\alpha_i(-n)$, $n\in\mathbb N, r\in\{1, \cdots, r\}$. Let $s_{\la}(\alpha_k)$ be the Schur function in the $\alpha_k(-n)$ \cite{IJS}.
Then the algebra $V$ admits an integral form $V_{\mathbb Z}=S_{\mathbb Z}(\hat{\mathfrak h}_{L})\otimes \mathbb Z\{L\}$ (see \cite{DG1}), where
$S_{\mathbb Z}(\hat{\mathfrak h}_{L})$ consists of the $\mathbb Z$-span of Schur functions indexed by partition-valued function $\underline{\lambda}=(\lambda^{(1)}(\alpha_1), \ldots, \lambda^{(r)}(\alpha_r))\in\mathcal P^r$.
The algebra naturally decomposes itself into
$$ V=\bigoplus_{n\in\mathbb Z_+, \alpha\in L}V_{n, \alpha}
$$
where $V_{n, \alpha}$ is spanned by $s_{\underline{\lambda}}e^{\alpha}$, and $e^{\alpha}\in\mathbb Z\{L\}$, the integral group algebra of $L$. So the general elements of the vertex operator algebra $V$ are divided into two parts: the components of $Y(v, z)$, $v\in V_{n, \alpha}, \alpha\neq 0$ and the components of
$Y(v, z)$, $v\in V_{n, 0}$. It is known that the operators of the first kind can generate the second kind \cite{Mc} in the vertex operator algebra.

For the {\it elements of the first kind} $v\in V_{n, \alpha}, \alpha\neq 0$, we will prove in great generality that the divided powers of integral vertex algebra elements from $Y(v, z)$
preserve the integral form of the lattice VOA, as long as the lattice is integral.
We will also show that the integral form of any irreducible module of the lattice VOA is also preserved by the lattice of the integral VOA
(the later lattice structure is in Lusztig's sense).
As a consequence, we recover Garland's result for the affine Lie algebras \cite{G} and their irreducible modules \cite{P}.
Then for any $v\in V_{n, \alpha}$ ($\alpha\neq 0$), $Y(v, z)=\sum_{n\in\mathbb Z}v_nz^{-n-1}$, we can define the invertible map:
\begin{equation}
\exp(v_nt)=\sum_{i=0}^{\infty}\frac{(v_n)^i}{i!}t^i
\end{equation}
which preserves the integral form $V_{\mathbb Z}$ as well as its  irreducible modules as $\mathbb Z$-spaces. Note that the automorphism is in general sense as
the infinitesimal group clearly does not preserve the grading.

For the {\it elements of the second kind} $v\in V_{n, 0}$, they are analogs of the Cartan subalgebra in the affine Lie algebra. By generalizing
Garland's operator \cite{G} we introduce the following symmetric functions: for any $k\in\mathbb N$, $\alpha\in$ basis of $L$
\begin{equation}
\exp(\sum_{n=1}\frac{\alpha(-kn)}{n}z^n)=\sum_{n=0}^{\infty}h_{-\alpha, n}^{[k]}z^n
\end{equation}
and we will show that $h_{-\alpha, n}^{[k]}$ preserves the integral lattice $(V_L)_{\mathbb Z}$.
When $L$ is the root lattice of ADE type, $V_L$ reduces to the basic module of the affine Lie algebra \cite{FLM}. In such a situation,
our proof also offers a different and simpler proof of Garland's results.

The paper is organized as follows. In Section 2, we discuss vertex operator algebras and symmetric functions,
and construct the integral forms
spanned by Schur functions. In Section 3, we derive some basic relations about vertex algebras and study the integral forms based on
integral lattices. An important technical lemma will be given, which generalizes a result of \cite{BFJ} and \cite{CJ}.
In Section 4, we will show that the divided powers of general vertex operators preserve the integral form
of the lattice vertex operator algebra as well as their irreducible modules $M_{\mathbb Z}$ over $\mathbb Z$. Moreover
the maps $\exp(v_nt)$ are in $\mathrm{GL}_{\mathbb Z}(V_{\mathbb Z})$ and $\mathrm{GL}_{\mathbb Z}(M_{\mathbb Z})$, and one can study modular vertex algebras
of lattice type and their irreducible modules.

\section{Vertex algebras and symmetric functions} 

We first recall the notion of the vertex algebra (VA).
A vertex algebra $V$ is a graded vector space $V=\oplus_{n\in\mathbb Z} V_n$ equipped with the state-field correspondence
map $Y_V(\cdot, z): V\longrightarrow \mathrm{End}(V)[[z, z^{-1}]]$, and $Y_V(u, z)=\sum_n u_nz^{-n-1}$ which defines
a sequence of vertex algebra products $u_nv$ for any $u, v\in V, n\in\mathbb Z$ that satisfies several defining
properties of the vertex algebras, see \cite{FLM, FHL} or \cite{LL} 
for details. In short, the vertex algebra $V$ is a non-associative
algebra equipped with
infinitely many products $u_nv$ that enjoy the local commutativity and associativity of the
vertex algebra.

An integral form $W_{\mathbb{Z}}$ of a finite-dimensional vector space $W$ is a $\mathbb Z$-span of its basis. If $V=\oplus_{n\in\mathbb Z} V_n$
is a graded algebra with finite dimensional homogeneous subspaces, then an integral form $V_{\mathbb Z}$ is a $\mathbb Z$-subspace $V_{\mathbb Z}$ such that $V_{\mathbb Z}\cap V_n$ are integral forms of $V_n$ and the product is closed under $V_{\mathbb Z}$.

For any nonnegative integer $r$ and an operator $x$, we define $x^{(r)}=\frac{x^r}{r!}$. Now let $U_{\mathbb Z}(V)$ be the $\mathbb Z$-linear span
of $(u_1)_{n_1}^{(r_1)}(u_2)_{n_2}^{(r_2)}\cdots (u_k)_{n_k}^{(r_k)}$, where $u_1, u_2, \cdots u_k\in V_{\mathbb Z}$ and $n_i\in \mathbb Z_+$.
In general, $U_{\mathbb Z}(V)$ may belong to the completion of the algebra and is the integral form of the vertex enveloping algebra $U(V)$ defined in \cite{FZ}.

\begin{defn} \rm
An integral form $V_{\mathbb{Z}}$ of a vertex algebra $V$ is an integral basis of $V$
as a graded vector space such that it contains the vacuum vector $\textbf{1}$ and is closed under vertex algebra products.
\end{defn}
Here the closeness means that the product of any two basis elements is a $\mathbb Z$-linear combination of the basis elements.
If the vertex algebra $V$ has a conformal element $\omega$, we will not require an integral form of $V$ to contain $\omega$ as it will
spoil many properties of the integral form.

Let $\mu=\{\mu_1, \mu_2, \cdots ,\mu_k\}$ be a sequence or composition of integers, $\lambda=(\lambda_1,\lambda_2,\cdots, \lambda_s)$ is called a subsequence of $\mu$ if its coordinates are part of those of $\{\mu_1, \mu_2, \cdots ,\mu_k\}$ (with the same order), the remaining part of
the sequence is denoted by $^c{\lambda}$ and called
the complementary subsequence. When $\mu_i$ are non-ascending sequence of integers: $\mu_1\geq \ldots\geq \mu_l>0$, then $\mu=(\mu_1, \ldots, \mu_l)$
is called a partition with length $l$. The weight of $\mu$ is defined to be $\sum_i\mu_i$. Sometimes we also arrange the parts of $\mu$ in ascending way $\mu=(1^{m_1}2^{m_2}\cdots)$, where
$m_i$ is the multiplicity of the part $i$ in $\mu$.

Let $L$ be an even non-degenerate integral lattice of rank $d$ provided with a nondegenerate symmetric $\mathbb Z$-bilinear form $\langle\cdot,\cdot\rangle$,  such that the set
$\Delta=\{ \ep_1,\ep_2,\cdots,\ep_d\}$ is a $\mathbb Z$-basis of $L$.
Let
$\mathfrak{h}=\mathfrak{h}_L$ be the complexified abelian Lie algebra $\mathbb C\otimes L$, and bilinear form
is naturally extended to $\mathfrak h$.  The Heisenberg algebra is the infinite-dimensional Lie algebra
$$\hat{\mathfrak h}_L=\coprod \limits_{n\in \mathbb{Z^{\times}}}\mathfrak{h}\otimes t^n\oplus \mathbb{C}k$$
with the commutation relation
\begin{equation}
[h_1(m), h_2(n)]=m\langle h_1, h_2\rangle\delta_{m, -n}k
\end{equation}
where we set $h(m)=h\otimes t^m$.

Let the twisted group algebra $\mathbb{C}\{L\}$ be the space generated by $e^{\alpha}$, $\alpha\in L$
with the multiplication
$$
e^{\alpha}e^{\beta}=\varepsilon(\alpha, \beta)e^{\alpha+\beta}
$$
where $\varepsilon( \ , \ )$ is the cocycle on $L$ defined by the central extension
\begin{equation*}
  1\rightarrow \langle \kappa | \kappa^2=1\rangle \hookrightarrow \hat{L} \rightarrow L \rightarrow 0,
\end{equation*}
so that $e^{\alpha}e^{\beta}=(-1)^{\langle \alpha, \beta\rangle}e^{\beta}e^{\alpha}$.
We remark that if $L$ is not even, one can choose a central extension of $L$ by the cyclic group
$\langle \kappa | \kappa^s=1\rangle$ so that $e^{\alpha}e^{\beta}=(-1)^{\langle \alpha, \beta\rangle+\langle \alpha, \alpha\rangle\langle \beta, \beta\rangle}e^{\beta}e^{\alpha}$. In the sequel, we will mainly work in the even case though many results
also hold in the super case.

View $\mathbb{C}\{L\}$ as a trivial $\hat{\mathfrak{h}}_{\mathbb{Z}}$-module and for $\alpha \in \mathfrak{h}$, define an action of $\mathfrak{h}$ on $\mathbb{C}\{L\}$ by:
$\alpha(0)\cdot e^\eta=\langle \alpha,\eta\rangle  e^\eta$, for $\eta \in L$.

Let $\hat{\mathfrak{h}}^{\pm}=\mathfrak{h}\otimes t^{\pm 1}\mathbb{C}[t^{\pm 1}]$
be the subalgebras of $\hat{\mathfrak h}$,
then $\hat{\mathfrak{h}}$ has the triangular decomposition:
$\hat{\mathfrak{h}}=\hat{\mathfrak{h}} ^+\oplus \hat{\mathfrak{h}}^-\oplus \mathbb{C} k$.
Let $S(\hat{\mathfrak{h}}^-)=M(1)$ be  the symmetric algebra generated by $\hat{\mathfrak{h}}^-$, which is naturally a
$\hat{\mathfrak{h}}$-module with $k=1$. Set
\begin{equation}
V_L=M(1)\otimes \mathbb{C}\{L\},
\end{equation}
and let $\mathbf{1}=1\otimes 1$ be the vacuum vector.
 Since both $S(\hat{\mathfrak{h}}^-)$ and $\mathbb{C}\{L\}$ are $\hat{\mathfrak{h}}$-module, we have  $V_L$ is an $\hat{\mathfrak{h}}$-module via the tensor product action.

The lattice vertex operator algebra $V_L$ \cite{FLM, LL}
is specified by the state-field correspondence $v\mapsto Y(v, z)$. 
For $\alpha \in \mathfrak{h}$, we define the operator $z^{\alpha(0)}$ on the space $V_L$:
$z^{\alpha(0)} \cdot (x\otimes e^\eta)=z^{\langle\alpha, \eta\rangle}(x\otimes e^\eta)$,
for $x\in S(\hat{\mathfrak{h}}_{\mathbb{Z}} ^-), \eta\in L$.
Define 
\begin{equation}\label{e:st-field}
  Y(e^\alpha,z)=E^-(-\alpha,z)E^+(-\alpha,z)e^\alpha z^{\alpha(0)}
\end{equation}
where
 \begin{align}
 E^+(-\alpha,z)&=\exp\bigg(-\sum\limits_{n\in {\mathbb{Z}_+}}\frac{\alpha(n)}{n}z^{-n}\bigg) , \\ \label{e:hom}
 E^-(-\alpha,z)&=\exp\bigg( \sum\limits_{n\in {\mathbb{Z}_+}}\frac{\alpha(-n)}{n}z^n\bigg)=
                  \sum\limits_{n\geq0}h_{\alpha,-n}z^n.
 \end{align}

For $\alpha \in \mathfrak{h}$ we set
$$\alpha(z)=\sum\limits_{n\in \mathbb{Z}}\alpha(n)z^{-n-1}, $$
and for the general basis element
$v=\alpha_1(-n_1)\cdots \alpha_k(-n_k)e^{\ga} \in V_L$,
$\gamma, \alpha_i \in \mathfrak{h},n_i,k\geq 1$, we define
\begin{align*}
  Y(v,z)&=\sum \limits_{n\in \mathbb{Z}}v_nz^{-n-1} \\
          &=\,: \left(\partial^{(n_1-1)}\alpha_1(z)
                        \cdots
                        \partial^{n_k-1}\alpha_k(z)\right)
                         Y(e^{\ga},z) :,
\end{align*}
where $\partial^{(n)}=\frac1{n!}(\frac{\partial}{\partial z})^n$ and $:\ \ :$ is the normal ordered product.

Let $L^{\circ}$ be the dual lattice of $L$, so the space
\begin{equation}
V_{L^{\circ}}=S(\hat{\mathfrak h}^-)\otimes \mathbb{C}\{L^{\circ}\}
\end{equation}
is similarly defined as in $V_L$ (see \cite{LL}) and
is naturally a $V_L$-module with the action given by the state-field map \eqref{e:st-field}.
Let $\{\gamma_i\}$ be the set of coset representatives of $L^{\circ}/L$. It is known \cite{D, LL} that $V_{L^{\circ}}$ decomposes into irreducible $V_{L}$-modules:
\begin{equation}
V_{L^{\circ}}=\bigoplus_{i=1}^{|L^{\circ}/L|}S(\hat{\mathfrak h}^-)\otimes \mathbb{C}\{L+\gamma_i\},
\end{equation}
and the irreducible components exhaust all irreducible $V_L$-modules.

For a partition $\lambda=(\lambda_1,\lambda_2,\cdots, \lambda_k)$,
define the homogeneous symmetric function $h_{\alpha, -\lambda}$ by
\begin{equation*}
h_{\alpha,-\lambda}=h_{\alpha,-\lambda_1}h_{\alpha,-\lambda_2}\cdots h_{\alpha,-\lambda_k}.
\end{equation*}
whose generating function is $E(-\alpha, z_1)\cdots E(-\alpha, z_k)$.

For partition $\lambda$ of length $l$, the Schur function $s_{\alpha, -\lambda}$ is defined by the Jacobi-Trudi formula
\begin{equation}
s_{\alpha, -\lambda}=\det(h_{\alpha, -\lambda_i+i-j})_{l\times l}
\end{equation}

It is well-known that $\alpha(-n)$ is an integral linear combination of $h_{\alpha, -\lambda}$, $|\lambda|=n$.
The homogenous function $h_{\alpha, -n}$ is an integral linear combination of the Schur functions $s_{\alpha, -\mu}$, where
$|\mu|=|\lambda|$, and vice versa.

Moreover, Schur functions can be created by vertex operators.
For $\alpha\in L$ we introduce the following vertex operator \cite{J2}
\begin{equation}
S(\alpha, z)=E^-(-\alpha, z)E^+(-\frac{\alpha}2, z)=\sum_{n\in\mathbb Z} S(\alpha)_n z^{-n}
\end{equation}

\begin{prop} \cite{J1} For each partition $\lambda=(\lambda_1, \ldots, \lambda_l)$, the
Schur function $s_{\lambda}$ in the variable $\alpha(-n)$, viewed as the power sum $p_n=\sum_i x_i^n$, is given by
\begin{equation}
s_{\alpha, -\lambda}=S(\alpha)_{-\lambda_1}S(\alpha)_{-\lambda_2}\cdots S(\alpha)_{-\lambda_l}.1
\end{equation}
\end{prop}

Let $\underline{\lambda}=(\lambda^{(1)}, \ldots, \lambda^{(d)})$ be a multi-partition with weight $n$, we define the
multivariate or tensor product Schur function $s_{\underline{\lambda}}$ \cite{IJS} as
$$
s_{\beta_1, -\lambda^{(1)}}s_{\beta_2, -\lambda^{(2)}}\cdots s_{\beta_m, -\lambda^{(m)}}.
$$
whose generating function is
\begin{equation}\label{e:genSchur1}
S(\beta_1, z_1)\cdots S(\beta_1, z_{l(\la^{(1)})})S(\beta_2, w_1)\cdots S(\beta_2, w_{l(\la^{(2)})})\cdots S(\beta_m, t_1)\cdots S(\beta_1, t_{l(\la^{(m)})}).1
\end{equation}
According to \cite{IJS}, the above is a $\mathbb Z$-linear combination of tensor products of homogeneous symmetric functions. Therefore we can use the following generating series for the tensor product Schur function:
\begin{equation}\label{e:genSchur2}
E^-(-\ga_1, z_1)E^-(-\ga_2, z_2)\cdots E^-(-\ga_l, z_l)
\end{equation}
where $\ga_i$ are a sequence of vectors (with multiplicity) in $L$.

\section{integral forms of lattice vertex operator algebras }  

Let $\underline{\lambda}=(\lambda^{(1)}, \ldots, \lambda^{(d)})$ be a multi-partition with weight $n$. The $\mathbb Z$-span
of Schur functions
$$s_{\beta_1, -\lambda^{(1)}}s_{\beta_2, -\lambda^{(2)}}\cdots s_{\beta_m, -\lambda^{(m)}}e^{\eta},$$
where $\beta_i\in\Delta$,
$\eta \in L$
can be simply written as a $\mathbb Z$-span of the following elements:
\begin{equation}\label{e:wrSchur}
s_{\beta_1',-n_1}s_{\beta_2',-n_2}\cdots s_{\beta_k',-n_k}\otimes e^\eta,
\end{equation}
where $\beta_i', \eta\in \Delta$ and $n=\sum_i{n_i}$ such that $n_1\geq n_2\geq \cdots \geq n_k\geq 0$. For convenience, we will stick this type of Schur elements
and note that the generating series is also given by appropriate products of the exponential operators
given in \eqref{e:genSchur2}.

Let $(V_L)_{\mathbb{Z}}$ be the $\mathbb Z$-span of
the elements in \eqref{e:wrSchur}. We also define $(V_{L^{\circ}})_{\mathbb Z}$ and $(V_{L+\gamma_i})_{\mathbb Z}$ similarly by the
same span (with $e^{\eta}$ being in $\mathbb C\{L^{\circ}\}$ or $\mathbb C\{L+\gamma_i\}$ respectively).
The following theorem is essentially from \cite{DG1} (see also \cite{Mc}).
\begin{thm}
The space $(V_L)_{\mathbb{Z}}$
is an integral form of $V_L$ (or $(V_{L+\gamma_i})_{\mathbb Z}$)) generated by $e^{\pm \alpha_i}$, $\alpha_i\in \Delta$.
Moreover, the space $(V_{L+\gamma_i})_{\mathbb Z}$
is an integral form of $V_{L+\gamma_i}$ generated by the action of $Y(e^{\pm \alpha_i}, z)$, $\alpha_i\in \Delta$
on the vacuum vector $e^{\gamma_i}$.
\end{thm}
\begin{proof} It is enough to show the theorem for $(V_L)_{\mathbb{Z}}$.
By the relation between Schur functions and the homogeneous symmetric functions, we see that the basis
 elements $s_{\beta_1,-n_1}s_{\beta_2,-n_2}\cdots s_{\beta_k,-n_k}\otimes e^\eta$, $\beta_i\in\Delta$
 can be expressed as integral linear combinations of $h_{\beta_1',-m_1}h_{\beta_2',-m_2}\cdots h_{\beta_l',-m_l}\otimes e^\eta$,
$\beta_i'\in\Delta$, where both partitions $(n_1, \ldots, n_k)$ and
 $(m_1, \ldots, m_l)$ are partitions of the same weight.
 The latter are the coefficients of $E^-(-\beta'_1,w_1)\cdots$ $E^-(-\beta'_l,w_l)\otimes e^\eta$.
 Then the result follows from \cite{DG1}.
\end{proof}

Let $\delta=(\delta_1,\delta_2,\cdots,\delta_r)=(r-1,r-2,\cdots,0)
\in \mathbb{Z}^r$ be the special partition, 
$l(\sigma)$ is the inverse number of the permutation $\sigma$. For any r-tuple $t=(t_1,t_2,\cdots, t_r), t_i\in \mathbb{Z}$,
let $R$ be a commutative ring, and let the space of truncated formal Laurent series be
$$R((z_1,z_2,\cdots,z_r))=
\bigg\{ \sum \limits_{t}a_tz^t=\sum \limits_{t}a_{t_1t_2\cdots t_r}z_1^{t_1}z_2^{t_2}\cdots z_r^{t_r}\Big|
a_{t_1t_2\cdots t_r}\in R, a_{t_1t_2\cdots t_r}=0 \ \mbox{for $t_i<<0$} 
\bigg\}.$$
The symmetric group $S_r$ acts on the formal Laurent series by permuting their variables, i.e.  for $\sigma\in S_r$,
 $$\sigma.f(z_1, \cdots, z_r)=f(z_{\sigma^{-1}(1)}, \cdots, z_{\sigma^{-1}(r)}), \qquad\quad f\in R((z_1,z_2,\cdots,z_r)).$$

\begin{lemma}\label{r!} If
 $G\in R((z_1,z_2,\cdots,z_r))$
is invariant under the action of the symmetric group $S_r$. Then for all $n \in \mathbb{Z}, k\in \mathbb{N}$, the coefficient of
$(z_1z_2\cdots z_r)^n$ in $\prod \limits_{1\leq i<j\leq r}(z_i-z_j)^kG$ is divisible by $r!$.
\end{lemma}
\begin{proof} If $k>3$, write $k=2q+k_0$ with $k_0=0,1$ and $q\geq 1$. Then
$\prod \limits_{1\leq i<j\leq r}(z_i-z_j)^{2q}G$ is clearly invariant under $S_r$. So it is enough to show the result for $k=1, 2$.

First of all,
\begin{align}\label{e:Van}
\prod\limits_{1\leq i<j\leq r}(z_i-z_j)=\sum\limits_{\sigma\in S_r}(-1)^{l(\sigma)}z^{\sigma(\delta)}
\end{align}
where $\sigma$ runs through all permutations of $S_r$, $\delta=(r-1, \ldots, 1, 0)$ and $z^{\mu}=z_1^{\mu_1}\cdots z_r^{\mu_r}$.
Then
\begin{align}\notag
\prod\limits_{1\leq i<j\leq r}(z_i-z_j)^2&=\sum\limits_{\sigma\in S_r}(-1)^{l(\sigma)}z^{\sigma(\delta)}\prod\limits_{1\leq i<j\leq r}(z_i-z_j)\\ \label{e:sq}
&=\sum\limits_{\sigma\in S_r}\sigma.\left(z^{\delta}\prod\limits_{1\leq i<j\leq r}(z_i-z_j)\right)
\end{align}
Note that if $(z_1\cdots z_r)^n$ appears in $\prod\limits_{1\leq i<j\leq r}(z_i-z_j)^2$, it must appear
inside $z^{\delta}\prod\limits_{1\leq i<j\leq r}(z_i-z_j)$, i.e. we can write
\begin{equation*}
z^{\delta}\prod\limits_{1\leq i<j\leq r}(z_i-z_j)=c_n(z_1\cdots z_r)^n+\sum_{\mu\neq (n, \cdots, n)}c_{\mu}z^{\mu}.
\end{equation*}
It follows from \eqref{e:sq} that
\begin{align*}
\prod\limits_{1\leq i<j\leq r}(z_i-z_j)^2&=\sum\limits_{\sigma\in S_r}\sigma.\left(c_n(z_1\cdots z_r)^n+\sum_{\mu\neq (n, \cdots, n)}c_{\mu}z^{\mu}\right)\\
&=r!c_n(z_1\cdots z_r)^n+\sum_{\mu\neq (n, \cdots, n), \sigma\in S_r}c_{\sigma^{-1}(\mu)}z^{\mu}.
\end{align*}
So the lemma is proved for $k=2$.

Now assume $k=1$. Suppose $(z_1\cdots z_r)^n$ appears inside $\prod \limits_{i<j}(z_i-z_j)G=\sum\limits_{\sigma\in S_r}(-1)^{l(\sigma)}z^{\delta}G$,
then the coefficient is determined by that of $(z_1\cdots z_r)^n$ in $z^{\delta}G$. We can write
\begin{equation}\label{e:sep}
z^{\delta}G=a_n(z_1\cdots z_r)^n+\sum_{\mu\neq (n, \cdots, n)}a_{\mu}z^{\mu}.
\end{equation}
where $a_{\mu}\neq 0$ for finitely many $\mu\in\mathbb Z^{r}$ with any fixed weight $|\mu|=\sum_i\mu_i$. Using \eqref{e:Van} and
\eqref{e:sep}, we have
\begin{align*}
\prod\limits_{1\leq i<j\leq r}(z_i-z_j)G&=\sum\limits_{\sigma\in S_r}(-1)^{l(\sigma)}\sigma.\left(a_n(z_1\cdots z_r)^n+\sum_{\mu\neq (n, \cdots, n)}a_{\mu}z^{\mu}\right)\\
&=\sum\limits_{\mu\neq (n, \cdots, n), \sigma\in S_r}(-1)^{l(\sigma)}a_{\sigma^{-1}(\mu)}z^{\mu},
\end{align*}
where the first summand vanishes due to anti-symmetry and the second summand contains no term $(z_1\cdots z_r)^n$, i.e.
the coefficient of $(z_1\cdots z_r)^n$ in $\prod\limits_{1\leq i<j\leq r}(z_i-z_j)G$ is zero. This completes the proof.
\end{proof}

\begin{thm}\label{CJ}
For $\alpha\in L$, let $Y(e^{\alpha},z)=\sum\limits_{n\in \mathbb{Z}}y_nz^{-n-1}$, then
$(V_L)_\mathbb{Z}$ (or $(V_{L+\gamma_i})_\mathbb{Z})$) is preserved by $y^{(r)}_n=\frac{y^r_n}{r!}$.
\end{thm}
\begin{proof} Recall that for $\alpha,\beta \in L$ \cite{FLM}
\begin{equation*}
  E^+(-\alpha,z)E^-(-\beta,w)=E^-(-\beta,w)E^+(-\alpha,z)(1-\frac{w}{z})^{\langle \alpha, \beta \rangle}
\end{equation*}
So we have,
\begin{align*}
Y(e^{\alpha},z_r)& \Big(E^-(-\beta_1,w_1)\cdots E^-(-\beta_l,w_l)\otimes e^\eta \Big)\\
    =&\epsilon (\alpha, \eta)z_r^{\langle \alpha, \eta \rangle}
      \prod\limits_{s=1}\limits^l(1-\frac{w_s}{z_r})^{\langle \alpha, \beta_s \rangle}
      E^-(-\alpha,z_r)E^-(-\beta_1,w_1)\cdots E^-(-\beta_l,w_l)\otimes e^{\alpha+\eta}.
\end{align*}
Then
\begin{align}\label{ynr}
  Y(e^{\alpha},z_1)&Y(e^{\alpha},z_2)\cdots Y(e^{\alpha},z_r)
    E^-(-\beta_1,w_1)\cdots E^-(-\beta_l,w_l)\otimes e^\eta \nonumber\\   \nonumber
    &=\epsilon \cdot (z_1z_2\cdots z_r)^{\langle \alpha, \eta \rangle}
       \prod\limits_{1\leq k< s\leq r}(z_k-z_s)^{\langle \alpha, \alpha \rangle}
       \prod\limits_{1\leq k\leq r,1\leq s\leq l}(1-\frac{w_s}{z_k})^{\langle \alpha, \beta_s \rangle}\\
       &\quad \cdot  E^-(-\alpha,z_1) \cdots E^-(-\alpha,z_r)E^-(-\beta_1,w_1)\cdots E^-(-\beta_l,w_l)\otimes e^{r\alpha+\eta},
\end{align}
where $\epsilon=\epsilon(\alpha,\eta)\epsilon(\alpha,\alpha+\eta)\cdots \epsilon(\alpha,(r-1)\alpha+\eta)$.
Subsequently using Lemma \ref{r!}, for $\alpha \in L, \langle \alpha, \alpha \rangle \in \mathbb{Z}_+$, we have that
\begin{equation*}
  y_n^{(r)}\cdot \bigg(E^-(-\beta_1,w_1)\cdots E^-(-\beta_l,w_l)\bigg)
             w_1^{\mu_1}w_2^{\mu_2}\cdots w_l^{\mu_l}\otimes e^{\eta}\in (V_L)_\mathbb{Z}
\end{equation*}
where $\mu_1,\mu_2,\cdots,\mu_l\in \mathbb{Z}, \eta \in \Delta$, which means that
$y_n^{(r)}(V_L)_\mathbb{Z} \subset (V_L)_\mathbb{Z}$.
\end{proof} 

\begin{corollary}\rm
For $\alpha_i \in L, k_i \in \mathbb{Z} $, let $v=\sum\limits_{i=1}\limits^m k_ie^{\alpha_i}$ such that
$\langle\alpha_i, \alpha_j\rangle\geq 0$, and $Y(v,z)=\sum\limits_{n\in \mathbb{Z}}v_nz^{-n-1}$, then
$(V_L)_\mathbb{Z}$ is preserved by $v^{(r)}_n=\frac{v^r_n}{r!}$.
\end{corollary}
\begin{proof} By vertex operator calculus it is well-known that if $\langle\alpha_i, \alpha_j\rangle\geq 0$, then
\begin{equation*}
  Y(e^{\alpha_1},z_1)Y(e^{\alpha_2},z_2)=Y(e^{\alpha_2},z_2)Y(e^{\alpha_1},z_1).
\end{equation*}
Also when the space $V$ is preserved under the action of the divided powers of commuting operators $A_1,\cdots, A_n$,
so is $A=\sum_i A_i$ under the action of $\mathbb Z$-linear combinations of product of divided powers of $A_i$.
In fact, this follows easily from $(A+B)^{(n)}=\sum_{i=1}^nA^{(i)}B^{(n-i)}$.
\end{proof} 

Before discussing the general vertex operator, let's consider the case of the Heisenberg algebra.
Recall that we have denoted that
$E^-(-\alpha,z)=\exp\bigg( \sum\limits_{n\in {\mathbb{Z}_+}}\frac{\alpha(-n)}{n}z^n\bigg)=
                  \sum\limits_{n\geq0}h_{\alpha, -n}z^n$. The right analog for the divided power in this case is
                  the following operator.
For any $k\in\mathbb N$, we introduce the Garland operators $h_{\alpha, -n}^{[k]}$ by
\begin{equation}
\exp\bigg( \sum\limits_{n\in {\mathbb{Z}_+}}\frac{\alpha(-kn)}{n}z^n\bigg)=\sum\limits_{n\geq0}h_{\alpha, -n}^{[k]}z^n.
\end{equation}
Special case of the operator $h_{\alpha, -n}^{[k]}$ was considered by Garland for affine Lie algebras \cite{G}.

\begin{thm}
For any $k, n\in \mathbb N$, the elements
$h_{\alpha, -n}^{[k]}\in \mathbb{Z}[h_{\alpha, -1},h_{\alpha, -2}, \cdots, h_{\alpha, -kn}]$.
In particular, as an operator $h_{\alpha, -n}^{[k]}$ preserves $(V_L)_\mathbb{Z}$.
\end{thm}
\begin{proof}
Let $\omega $ be the kth roots of unity: $\omega^k=1$, then
$$E^-(-\alpha, z)E^-(-\alpha, z\omega)\cdots E^-(-\alpha, z\omega^{k-1})=\exp\bigg( \sum\limits_{n\in {\mathbb{Z}_+}}\frac{\alpha(-kn)}{n}z^{kn}\bigg)
=\sum\limits_{n\geq0}h_{\alpha, -n}^{[k]}z^{kn}, $$
since $1+\omega^j+\omega^{2j}+\cdots+\omega^{(k-1)j}=0$, where $k\nmid j$.
Taking the coefficients of $z^{kn}$ ($n>0$), we have
 $$h_{\alpha, -n}^{[k]}=\sum\limits_{i_1+i_2+\cdots+i_k=kn}h_{\alpha, i_1}h_{\alpha, i_2}\cdots h_{\alpha, i_k}\omega^{i_2+2i_3+\cdots(k-1)i_k}
 \in \mathbb{Z}[\omega][h_{\alpha, -1},h_{\alpha, -2}, \cdots, h_{\alpha, -kn}].$$
It is obvious that $h_{-n}^{[k]} \in \mathbb{Q}[\alpha(-1),\cdots \alpha(-kn)] = \mathbb{Q}[h_{\alpha, -1},h_{\alpha, -2}, \cdots, h_{\alpha, -kn}]$.
Therefore
$$h_{\alpha, -n}^{[k]} \in (\mathbb{Q}\cap \mathbb{Z}[\omega])[h_{\alpha, -1},h_{-2}, \cdots, h_{\alpha, -kn}]=\mathbb{Z}[h_{\alpha, -1},h_{\alpha, -2}, \cdots, h_{\alpha, -kn}].$$
\end{proof} 

\section{$\mathbb Z$-lattice structure of the lattice vertex algebra}  

Now we consider the action of general homogeneous vertex operator on $(V_L)_{\mathbb{Z}}$.
Let $v=\alpha_1(-\la_1)$ $\cdots$ $\alpha_k(-\la_k)e^{\ga} \in V_L$,
$\alpha_i, \ga\in L^{\times}$, where $\la_i\geq 1, k\geq 0$, here $k=0$ means $v=e^{\ga}$.

For $\beta\in L$, we write $\beta(z)=\beta^+(z)+\beta^-(z)$, where $\beta^{\pm}(z)$ refers to the annihilation/creating part. Then for
$m\geq 1$
\begin{align*}
\partial^{(m-1)}\beta(z)
&=A^-_{\beta, m}(z)+A^+_{\beta, m}(z)
\end{align*}
where $A^+_{\beta, m}(z)=\sum\limits_{n\geq 0}(-1)^{m-1}\binom{n+m-1}{m-1}\beta(n)z^{-n-m}$
and
$A^-_{\beta, m}(z)=\sum\limits_{n>0}\binom{n-1}{m-1}\beta(-n)z^{n-m}$
are the annihilation and creation parts of the operator, respectively. 
We omit the subscript $\beta$ if no confusion arises.

For $\al(\la)=\alpha_1(-\la_1)\cdots \alpha_k(-\la_k)$, 
let
\begin{align*}
A^+_{\al(\la)}(z):&=A^+_{\al_1, \lambda_1}(z)\cdots A^+_{\al_k, \lambda_k}(z),\\
A^-_{\al(\la)}(z):&=A^-_{\al_1, \lambda_1}(z)\cdots A^-_{\al_k, \lambda_k}(z).
\end{align*}
where the index of $\la_i$ matches with that of $\al_i$. We can view $\la=(\la_1, \ldots, \la_k)$ as
a parition, and also view $(\la_1, \ldots, \la_k)$ as a decomposition of $\sum_i\al_i$, thus the notation
$|\al|=\sum_i\al_i$ will also be adopted. When $\al_i$ are fixed and omitted, we will use $\la_i$ to specify
the dependence of $\al_i$.
Then for $v=\al_1(-n_1)\cdots \al_k(-n_k)e^{\ga}$ we can write
\begin{align}\label{Yvz}
  Y(v,z)&=\sum \limits_{n\in \mathbb{Z}}v_nz^{-n-1}  \nonumber \\
          &=\, : \partial^{(n_1-1)}\alpha_1(z)
                        \cdots
                        \partial^{(n_k-1)}\alpha_k(z)\cdot
                         Y(e^{\ga},z) : \nonumber \\
          &=\sum \limits_\lambda A^-_{\al(\lambda)}(z)Y(e^{\ga},z)A^+_{\al(^c{\lambda})}(z),
\end{align}
summed over $2^k$ subpartitions $\lambda=(\lambda_1,\cdots, \lambda_s)$ of $(n_1, \cdots ,n_k)$, and $^c{\lambda}$ is the complementary subpartition of $\la$ inside $(n_1, \cdots ,n_k)$. i.e. if $\la=(n_{i_1}, \cdots, n_{i_s})$, then
$\al(\la)=\al_{i_1}(-n_{i_1})\cdots \al_{i_1}(-n_{i_1})$.

Fix $\al_1, \ldots, \al_k\in L$. For a sequence $\underline{\beta}=(\beta_1,\beta_2,\cdots,\beta_l)$, $\beta_i\in L$
and $\al(-m)$, $m\in\mathbb Z_+$, we define
\begin{equation}\label{fz1}
f_{\al, m}(z; w):=f_{\al, m}(\underline{\beta}, z; w)
\sum_{i=1}^l\frac{\langle\al, \beta_i\rangle(-1)^{m-1}}{(z-w_i)^{m}}
\end{equation}
where $w=(w_1, \cdots, w_l)$ and the rational function refers to the power series in $w_i$.
Also for fixed $\underline{\beta}$ and $\al(\la)=\al_1(-\la_1)\cdots \al_k(-\la_k)$ we define the formal series:
\begin{equation}\label{fzz}
f_{\al(\lambda)}(z,w):=f_{\al(\lambda)}(\underline{\beta}, z; w)=f_{\al_1, \lambda_1}(\underline{\beta}, z; w)\cdots f_{\al_k, \lambda_k}(\underline{\beta}, z; w)
\end{equation}

\begin{lemma} \label{A1E}
Let $\al(\lambda)=\al_1(-\la_1)\al_2(-\lambda_2)\cdots \al_k(-\la_k)$. Then for any sequence $(\beta_1,\beta_2,\cdots,\beta_l), \beta_q\in L, 1\leq q\leq l$, and $|\beta|=\sum_{i=1}^{l}\beta_i$, one has that
\begin{align}\notag
 A^+_{\al(\lambda)}(z)\prod_{i=1}^lE^-(-\beta_i,w_i)e^{|\beta|}
 &=\prod_{j=1}^k\left(\sum_{i=1}^l\frac{\langle\al_j, \beta_i\rangle(-1)^{\la_j-1}}{(z-w_i)^{\la_j}}\right)\prod_{i=1}^lE^-(-\beta_i,w_i)e^{|\beta|}\\ \label{e:AlE}
 &=f_{\al(\la)}(\beta,z; w)\prod_{i=1}^lE^-(-\beta_i,w_i)e^{|\beta|}.
 \end{align}
\end{lemma}
\begin{proof} Note that 
$[\alpha(n),E^-(-\beta,w)e^{\beta}]=\langle \alpha, \beta \rangle w^nE^-(-\beta,w)$ ($n\geq 0$) and
$[\alpha(0), e^{\eta}]=\langle\alpha, \eta\rangle e^{\eta}$.
It follows that
\begin{flalign}\label{AE}
[A^+_{\al, m}(z),E^-(-\beta,w)e^{\beta}]&=\langle \alpha, \beta \rangle  \bigg(\sum\limits_{n\geq 0}\binom{n+m-1}{m-1}(-1)^{m-1}z^{-n-m} w^n \bigg) E^-(-\beta,w)e^{\beta} \nonumber\\
&=(-1)^{m-1}\langle \alpha, \beta\rangle(z-w)^{-m}E^-(-\beta,w)e^{\beta}
\end{flalign}
and $A^+_{\al, m}(z)e^{\eta}=\langle\al, \eta\rangle(-1)^{m-1}z^{-m}e^{\eta}$.
Using these relations we get that
\begin{align*}
&A^+_{\al, m}(z)E^-(-\beta_1,w_1)\cdots E^-(-\beta_l,w_l).e^{\eta}=\epsilon A^+_{\al, m}(z)E^-(-\beta_1,w_1)\cdots E^-(-\beta_l,w_l)e^{\sum_i\beta_i}.e^{\eta-\sum_i\beta}\\
&=\epsilon \prod_{i=1}^l(E^-(-\beta_i,w_i)e^{\beta_i})
A^+_{\al, m}(z).e^{\eta-\sum_i\beta_i}+\epsilon\left(\sum_{i=1}^l\frac{\langle\al, \beta_i\rangle(-1)^{m-1}}{(z-w_i)^{m}}\right)\prod_{i=1}^lE^-(-\beta_i,w_i)e^{\beta_i}.e^{\eta-\sum_i\beta_i}\\
&=\left((-1)^{m-1}z^{-m}\langle\al, \eta-\sum_{i=1}^l\beta_i\rangle+\sum_{i=1}^l\frac{\langle\al, \beta_i\rangle(-1)^{m-1}}{(z-w_i)^{m}}\right)
\prod_{i=1}^lE^-(-\beta_i,w_i)e^{\eta},
\end{align*}
where $\epsilon=\varepsilon(|\beta|, \eta-|\beta|)^{-1}$.
Successively applying $A^+_{\al, \la_j}(z_j)$, we obtain that
\begin{align*}
 A^+_{\al(\lambda)}(z)&\prod_{i=1}^lE^-(-\beta_i,w_i)e^\eta=\bigg(A^+_{\al_1, \lambda_1}(z)\cdots A^+_{\al_k,\lambda_k}(z)\bigg)\prod_{i=1}^lE^-(-\beta_i,w_i)e^\eta\\
 &=\prod_{j=1}^k\left((-1)^{\la_j-1}z^{-\la_j}\langle\al_j, \eta-|\beta|\rangle+\sum_{i=1}^l\frac{\langle\al_j, \beta_i\rangle(-1)^{\la_j-1}}{(z-w_i)^{\la_j}}\right)\prod_{i=1}^lE^-(-\beta_i,w_i)e^\eta.
\end{align*}
In particular, \eqref{e:AlE} follows by taking $\eta=|\beta|=\sum_i\beta_i$.
\end{proof}

This lemma allows us to move the annihilation operator $A^+_{\al(\lambda)}(z)$ to the right. Next we consider how to move it
across the creating operator $A^-_{\al(\xi)}(z)$ and $Y(\iota(e_\alpha),z_j)$.

\begin{lemma} \label{A1A0}
Let $\lambda=(\lambda_1,\ldots, \lambda_k), \mu=(\mu_1,\ldots, \mu_l)$ be two compositions of nonnegative integers, and
$\al_1, \ldots, \al_k, \beta_1, \ldots, \beta_l\in L$, then
\begin{equation}\label{a+a-}
  A^+_{\al(\lambda)}(z)A^-_{\al(\mu)}(w)= \sum\limits_{\mu^\ast, \lambda^\ast}C_{\mu^\ast, \lambda^\ast}(z,w)
                               A^-_{\al(\mu^\ast)}(z)A^+_{\al(\lambda^\ast)}(w),
\end{equation}
where $\mu^\ast, \lambda^\ast$ run through all possible pairs of equal length subcompositions of $\mu$ and permutations of subcompositions of $\lambda$ respectively and
$C_{\xi^\ast, \lambda^\ast}(z,w) \in \mathbb Z[[z,z^{-1},w,w^{-1}]]$.
\end{lemma}
\begin{proof} For $\al, \beta\in L$, it is easy to see that for $m, n\geq 1$
\begin{equation*}
[A^+_{\al, m}(z), A^-_{\beta, n}(w)]=(-1)^{m-1}\langle\al, \beta\rangle\sum_{k=1}^{\infty}k\binom{k+m-1}{m-1}\binom{k-1}{n-1}z^{-k-m}w^{k-n}\in\mathbb
Z[[z^{-1}, w]]w^{-n+1}.
\end{equation*}

So the contraction function between the annihilation field $A^+_{\al, m}(z)$ and creation
field $A^-_{\beta, n}(w)$ is
\begin{equation}\label{e:C1}
C(A^+_{\al, m}(z), A^-_{\beta, n}(w))=\frac{(-1)^{m-1}\langle\al, \beta\rangle}{z^{m}w^{n}}\sum_{i=1}^{\infty}i\binom{i+m-1}{m-1}\binom{i-1}{n-1}\left(\frac{w}{z}\right)^i.
\end{equation}
Now for $\al(\la)=\al_1(-\la_1)\cdots \al_k(-\la_k)$ and $\beta(\mu)=\beta_1(-\mu_1)\cdots \beta_l(-\mu_k)$, the function
\begin{equation}\label{e:C2}
C(A^+_{\al(\la)}(z), A^-_{\beta(\mu)}(w))=\prod_{i=1}^kC(A^+_{\al_i, \la_i}(z), A^-_{\beta_i, \mu_i}(w))
\end{equation}
is the product of contraction functions between the two ordered sets of annihilation operators\newline $\{A^+_{\al_i, \la_i}(z)\}_{i=1}^k$ and the
creation operators $\{A^-_{\beta_j, \mu_j}(w)\}_{j=1}^l$.  Using Wick's theorem (cf. \cite{K}) we have that
\begin{align*}
A^+_{\al(\lambda)}(z)A^-_{\beta(\mu)}(w)&=A^+_{\al_1,\lambda_1}(z)\cdots
                  A^+_{\al_k,\lambda_k}(z_i)A^-_{\beta_1, \mu_1}(w)\cdots A^-_{\beta_l, \mu_l}(w)\\
              &=\sum_{\la^*, \mu^*}C(A^+_{\al(\la^*)}(z), A^-_{\beta(\mu^*)}(w))A^-_{\beta(\overline{\mu^*})}(w)A^+_{\al(\overline{\la^*})}(z)
\end{align*}
summed over all possible paired subcompositions $\la^*$ of $\la$ and permutations $\mu^*$ of subcompositions of $\la$
with the same length, and $C(\ \ , \ \ )$ is defined in \eqref{e:C1}-\eqref{e:C2}. Explicitly one first selects a subcomposition $\la^*$ of $i$ parts and then pairs it with permutations of
any $i$-part subcomposition $\mu^*$, i.e., there are $\sum_{i\geq 0}\binom{k}{i}\binom{l}{i}i!$ such pairings. Also $^c{\tau}$ is the complementary subpartition of $\tau$ in the partition. In particular, when $\la^*=\mu^*=\emptyset$, the summand is
$A^-_{\beta(\mu)}(w)A^+_{\al(\la)}(z)$.
\end{proof}

The next lemma considers the commutation relation between $A^+_{\al(\lambda)}(z)$ with  $Y(e^{\beta},w)$.
\begin{lemma}  \label{A1Y}
Let $\lambda=(\lambda_1,\ldots \lambda_k) \in (\mathbb{Z}^+)^k, \alpha_i, \beta\in L$, then
\begin{align}\label{AY1}
A^+_{\al(\lambda)}(z)Y(e^{\beta},w)&=Y(e^{\beta}, w)
            \prod \limits_{p=1}\limits^k \left(A^+_{\al_p, \lambda_p }(z)
         +(-1)^{\lambda_p-1}\langle \alpha_p, \beta \rangle(z-w)^{-\lambda_p}
         \right),\\ \label{AY2}
Y(e^{\beta},w)A^-_{\al(\lambda)}(z)&=\prod \limits_{p=1}\limits^k \bigg(A^-_{\al_p,\lambda_p}(z)-\langle\alpha_p, \beta \rangle
(w-z)^{-\la_p}\bigg)Y(e^{\beta},w),
\end{align}
where the rational functions refer to power series in $w$ in the first relation and $z$ in the second one.
\end{lemma}
\begin{proof} It follows from \eqref{AE} that
\begin{align*}
  A^+_{\al_p, \lambda_p }(z)E^-(-\beta, w)e^{\beta}&
=E^-(-\beta,w)e^{\beta}\bigg(A^+_{\al_p, \lambda_p }(z)
         +(-1)^{\lambda_p-1}\langle \alpha_p, \beta \rangle(z-w)^{-\lambda_p}\bigg),
\end{align*}
where the rational function is expanded at $w$. Similarly we also have
\begin{equation*}
E^+(-\beta, w)A^-_{\al_p,\lambda_p}(z)=\left(A^-_{\al_p,\lambda_p}(z)-\frac{\langle \alpha_p, \beta \rangle}{(w-z)^{\la_p}}\right)E^+(-\alpha,w).
\end{equation*}
where the rational function $(w-z)^{-\la_p}$ is expanded as a powers series in $z$.
The lemma is then proved by repeated applying $A^+_{\al_p, \lambda_p }(z)$ or $A^-_{\al_p, \lambda_p }(z)$ as above.
\end{proof}

We now prove our first main result of this paper.
\begin{theorem} The integral form $(V_L)_{\mathbb Z}$ and all of its irreducible modules $(V_{L+\gamma_i})_{\mathbb Z}$
associated with the vertex operator algebra $V_L$ are preserved by the divided powers of the general vertex operator $Y(v, z)$, where
$v=\alpha_1(-n_1)\alpha_2(-n_2)\cdots \alpha_k(-n_k)e^{\ga} \in V_L, \ga\neq 0$.
In particular, $s_{\alpha(\la)}e^{\beta}$ (resp. $s_{\alpha(\la)}e^{\beta+\gamma_i}$)
span a $\mathbb Z$-lattice for the vertex operator algebra $(V_L)_{\mathbb Z}$ (resp. its irreducible module $(V_{L+\gamma_i})_{\mathbb Z}$).
\end{theorem}
\begin{proof}
It's enough to consider the case of $(V_L)_{\mathbb Z}$. Write $v=\al_1(-\la_1)\cdots \al_k(-\la_k)e^{\ga}=\alpha(\la)e^{\ga}$, then
\begin{align}\notag
Y(v, z)&=\sum_nv_nz^{-n-1}=:\partial^{(\la_1-1)}\al_1(z)\cdots \partial^{(\la_k-1)}\al_k(z)Y(e^{\ga}, z):\\ \label{e:Yv}
&=\sum_{\la^*}A^-_{\al(\la^*)}(z)Y(e^{\ga}, z)A^+_{\al(\bar{\la}^*)}(z)
\end{align}
summed over all subpartitions $\la^*$ of $\la$.

 Let $E(w)=E^-(-\beta_1,w_1)\cdots E^-(-\beta_l,w_l)\otimes e^{\eta} $. It is enough to consider $\eta=\sum_i\beta_i=|\beta|$,
 so $E(w)=E^-(-\beta,w)e^{|\beta|}$, where $\beta=(\beta_1, \ldots, \beta_l), w=(w_1, \ldots, w_l)$.
 By Lemma \ref{A1E} and \eqref{e:Yv}
\begin{align}\label{e:YE1}\notag
  Y(v,z)E(w)&
=\sum \limits_{\lambda^*} A^-_{\al(\lambda^*)}(z)Y(e^{\ga},z)A^+_{\al(\bar{\lambda}^*)}(z)E(w)  \\ \notag
&=\sum \limits_{\lambda^*} f_{\al(\bar{\la}^*)}(\beta, z; w)A^-_{\al(\lambda^*)}(z)Y(e^{\ga},z)E(w)\\
&=\epsilon(\ga, |\beta|)\sum_{\la^*}f_{\al(\bar{\la}^*)}(\beta, z; w)A^-_{\al(\lambda^*)}(z)\prod_{i}(z-w_i)^{\langle\ga, \beta_i\rangle}E^-(-(\beta,\ga), w, z)e^{|\beta|+\ga}\\ \notag
&=\sum_{\la^*}F_{\alpha(\la^*)}(z, w)A^-_{\al(\lambda^*)}(z)E^-(-(\beta,\ga), w, z)e^{|\beta|+\ga}
\end{align}
where $F_{\alpha(\la^*)}(z, w)\in \mathbb Z[[z, z^{-1}, w_i, w_i^{-1}]]$
as $f_{\al(\bar{\la}^*)}(\beta, z, w)$ is defined in \eqref{fzz}.
Note that
\begin{align*}
&Y(v, z_2)A^-_{\al(\lambda^*)}(z_1)E^-(-(\beta,\ga), w,z_1)e^{|\beta|+\ga}\\
&=\sum_{{\la^{(1)}}^*}\left(A^-_{\al({\lambda^{(1)}}^*)}(z_2)Y(e^{\ga},z_2)A^+_{\al({^c\lambda^{(1)}}^*)}(z_2)\right)
A^-_{\al(\lambda^*)}(z_1)E^-(-(\beta, \ga), w, z_1)e^{|\beta|+\ga}\\
&=\sum_{{\la^{(1)}}^*}C_{\al({^c\lambda^{(1)}}^*), \al(\lambda^*)}(z_2, z_1)A^-_{\al({\lambda^{(1)}}^*)}(z_2)Y(e^{\ga},z_2)
A^-_{\al(\lambda^*)}(z_1)A^+_{\al({^c\lambda^{(1)}}^*)}(z_2)E^-(-(\beta,\ga), z_1\cup w)e^{|\beta|+\ga}\\
&=\sum_{{\la^{(1)}}^*}C_{\al({^c\lambda^{(1)}}^*), \al(\lambda^*)}(z_2, z_1)A^-_{\al({\lambda^{(1)}}^*)}(z_2)Y(e^{\ga},z_2)
A^-_{\al(\lambda^*)}(z_1)\\
&\hskip 5cm \cdot f_{\al({^c\lambda^{(1)}}^*)}((\beta, \ga), z_2; w, z_1)E^-(-(\beta,\ga), w, z_1)e^{|\beta|+\ga}.
\end{align*}
Recalling Lemma \ref{A1Y} and \eqref{e:YE1}, the above can be written as:
\begin{align*}
  &Y(v, z_2)Y(v,z_1)E\\
&=\sum_{\la^*, {\la^{(1)}}^*}F_{\la^*, {\la^{(1)}}^*}(z_1, z_2, w)(z_1-z_2)^{\langle \ga, \ga\rangle}A^-_{\al({\lambda^{(1)}}^*),\al(\lambda^*)}(z_1,z_2)
E^-(-(\beta,\ga,\ga), w, z_1, z_2)e^{|\beta|+2\ga}
\end{align*}
for some series $F_{\la^*, {\la^{(1)}}^*}(z_1, z_2, w)\in\mathbb Z[[z_i, z_i^{-1}, w_j, w_j^{-1}]]$.
Continuing in this way, we have
\begin{align*}
  &Y(v, z_r)\cdots Y(v, z_2)Y(v,z_1)E\\
&=\sum_{\la^*, {\la^{(1)}}^*, \ldots, {\la^{(r-1)}}^*}F_{\la^*, \ldots, {\la^{(r-1)}}^*}(z, w)\prod_{1\leq i<j\leq r}(z_i-z_j)^{\langle \ga, \ga\rangle}A^-_{\al({\lambda^{(r-1)}}^*),\ldots, \al(\lambda^*)}(z)
E^-(-(\beta,\ga^r), w, z)e^{|\beta|+r\ga}
\end{align*}
where $z=(z_1, \ldots, z_r), w=(w_1, \dots, w_l)$, and $F_{\la^*, \ldots, {\la^{(r-1)}}^*}(z, w)$ are some series in $\mathbb Z[[z_i, z_i^{-1}, w_j, w_j^{-1}]]$. The sum
runs through sequences of subpartitions $\la^*, {\la^{(1)}}^*, \ldots, {\la^{(r-1)}}^*$.

By a result in \cite{IJS}, for any vector $\al(\la)=\al_1(-\la_1)\cdots \al_k(-\la_k)$, the product $\al(\la)s_{\beta, -\mu}$
is still a $\mathbb Z$-linear combination of tensor product Schur functions. It follows that
$$A^-_{\al({\lambda^{(r-1)}}^*),\ldots, \al(\lambda^*)}(z)
E^-(-(\beta,\ga^r), w, z)e^{|\beta|+r\ga}$$
is also a $\mathbb Z$-linear combination of tensor product Schur functions
in view of our vertex operator realization \eqref{e:genSchur1}-\eqref{e:genSchur2}.
Finally it follows from Lemma \ref{r!} that the coefficient of $\frac{(z_1\cdots z_r)^{n+1}}{r!}$ in
$$Y(v, z_r)\cdots Y(v, z_2)Y(v,z_1)E$$
is always a $\mathbb Z$-linear combination of wreath product Schur functions, i.e. an element in the
$\mathbb Z$-form of the vertex operator algebra $V_L$.
\end{proof}

As $V$ is locally finite, for each vertex operator $Y(v, z)=\sum_{n\in\mathbb Z}v_nz^{-n-1}$, we can define
\begin{equation}
\exp(tv_n)=\sum_{r=0}^{\infty} v_n^{(r)}t^r
\end{equation}
as an element of the linear group $\mathrm{GL}((V_L)_{\mathbb Z})$ or $\mathrm{GL}((V_{L+\gamma_i})_{\mathbb Z})$. By our main theorem, the element
$\exp(tv_n)u$ is in $(V_L)_{\mathbb Z}$ or $(V_{L+\gamma_i})_{\mathbb Z}$. So we can summarize:

\begin{theorem} Let $L$ be an integral lattice over $\mathbb Z$.
For any homogeneous $v\in (V_L)_{\mathbb Z}$ with nontrivial lattice part
and $Y(v, z)=\sum_{n\in\mathbb Z}v_nz^{-n-1}$, the operator $\exp(tv_n)$
generates an element in the automorphism group $\mathrm{GL}((V_L)_{\mathbb Z})$ (or $\mathrm{GL}((V_{L+\gamma_i})_{\mathbb Z})$).
\end{theorem}

Let $\mathbb F_q$ be a fixed finite field with characteristic $p$, the lattice vertex algebra $V_q$ over $\mathbb F_q$ associated with $V_L$ is
usually defined as $
V_q=\mathbb F_q\otimes (V_L)_{\mathbb Z},
$
which is known to be a simple vertex algebra when $\det(L)\neq 0$ in $\mathbb F_q$ (cf. \cite{Mu, DG2}).

\vskip30pt \centerline{\bf Acknowledgments}
The work is partially supported by
Simons Foundation grant No. 523868. 

\end{document}